\documentclass{article}

\usepackage{amsmath,amssymb,amsthm,graphics,tikz}
\usepackage{comment}

\title{Prenex normalization and the hierarchical classification of formulas}
\author{Makoto Fujiwara\footnote{Email: makotofujiwara@rs.tus.ac.jp}
\footnote{Department of Applied Mathematics, Faculty of Science Division I, Tokyo University of Science, 1-3 Kagurazaka, Shinjuku-ku, Tokyo 162-8601, Japan.}
and Taishi Kurahashi\footnote{Email: kurahashi@people.kobe-u.ac.jp}
\footnote{Graduate School of System Informatics,
Kobe University,
1-1 Rokkodai, Nada, Kobe 657-8501, Japan.}}
\date{}

\theoremstyle{plain}
\newtheorem{thm}{Theorem}[section]
\newtheorem*{thm*}{Theorem}
\newtheorem{lem}[thm]{Lemma}
\newtheorem{prop}[thm]{Proposition}
\newtheorem{cor}[thm]{Corollary}

\theoremstyle{defn}
\newtheorem{defn}[thm]{Definition}

\newtheorem{remark}[thm]{Remark}

\newcommand{\FV}{\mathrm{FV}}
\newcommand{\DT}{\rightsquigarrow}
\newcommand{\DTA}{{\rightsquigarrow^*}}
\newcommand{\PNF}{\mathsf{PNF}}

\newcommand{\degree}[1]{\mathit{deg}(#1)}
\newcommand{\alt}[1]{\mathit{Alt}(#1)}

\newcommand{\PA}{\mathsf{PA}}
\newcommand{\HA}{\mathsf{HA}}

\newcommand{\F}{\mathrm{F}}
\newcommand{\PF}{\mathrm{P}}
\newcommand{\U}{\mathrm{U}}
\newcommand{\E}{\mathrm{E}}

\newcommand{\vp}{\varphi}
\newcommand\ang[1]{{\langle #1 \rangle} }

\begin{document}

\maketitle

\begin{abstract}
Akama et al.~\cite{ABHK04} introduced a hierarchical classification of first-order formulas for a hierarchical prenex normal form theorem in semi-classical arithmetic.
In this paper, we give a justification for the hierarchical classification in a general context of first-order theories.
To this end, we first formalize the standard transformation procedure for prenex normalization.
Then we show that the classes $\E_k$ and $\U_k$ introduced in \cite{ABHK04} are exactly the classes induced by $\Sigma_k$ and $\Pi_k$ respectively via the transformation procedure in any first-order theory.
\end{abstract}

\section{Introduction}
We study the prenex normalization of first-order formulas by the standard reduction procedure without any reference to the notion of derivability.
The prenex normal form theorem states that for any first-order theory based on classical logic, every formula is equivalent (over the theory in question) to some formula in prenex normal form (cf.~\cite[pp.~160--161]{End01}).
This theorem is verified by using the fact that several transformations of formulas moving quantifiers in the formula from inside to outside in a suitable way preserve the validity with respect to first-order classical logic (cf.~\cite[pp.~37--38]{Sho01}). 
For example, if $x$ is not contained in $\delta$, then $\forall x \xi(x) \to \delta$ is transformed into $\exists x(\xi(x) \to \delta)$ with preserving classical validity because they are classically equivalent.
For each first-order formula, one can obtain an equivalent formula in prenex normal form by the following procedure:
\begin{enumerate}
\item
Apply the above mentioned transformations finitely many times to the subformulas of the form $A \circ B$ with $A$ and $B$ in prenex normal form where $\circ \in \{ \land, \lor, \to \}$, and transform  the subformulas into equivalent formulas in prenex normal form;
\item
Repeating this procedure until when all  subformulas become to be in prenex normal form.
\end{enumerate}

In contrast, the prenex normal form theorem does not hold for intuitionistic theories. 
For example, $(\forall x \xi(x) \to \delta) \to \exists x(\xi(x) \to \delta)$ is not provable in intuitionistic logic, and then the above procedure does not yield  an intuitionistically equivalent formula in prenex normal form.
Therefore the classical hierarchy of $\Sigma_k$ and $\Pi_k$ formulas, which is based on prenex formulas, does not make sense for intuitionistic theories.
Based on this fact, for intuitionistic theories, several kinds of hierarchical classes corresponding to $\Sigma_k$ and $\Pi_k$ have been introduced and studied from different perspectives respectively.
Some hierarchical classes were studied from the perspective of decidability and computational complexity (cf. \cite{Mints68, SUZ17}), some others were from the perspective of syntactic preservation theorems with respect to Kripke semantics (cf. \cite{Fle10}), and some others (for arithmetic) were from the perspective of proof-theoretic strength (cf. \cite{Burr00}).
A related work can be found in \cite{Leiv81}.
In addition, another approach has been developed recently in \cite{BNI19}.

Among these attempts, Akama, Berardi, Hayashi and Kohlenbach \cite{ABHK04} introduced the classes $\E_k$ and $\U_k$ of formulas corresponding to $\Sigma_k$ and $\Pi_k$ respectively, and argued that a hierarchical prenex normal form theorem for these classes holds for certain theories of semi-classical arithmetic. 
Their classes $\E_k$ and $\U_k$ are non-cumulative, that is, $\E_{k'}$ is not a subclass of $\E_k$ for $k' < k$. 
In \cite{FK21}, the authors introduced the cumulative variants $\E_k^+$ and $\U_k^+$, and corrected the hierarchical prenex normal form theorem argued in \cite{ABHK04} as follows (cf.~\cite[Theorem 5.3]{FK21}): for a $\HA$-formula $\vp$, 
if $\varphi \in \E_k^+$, then there exists a $\varphi' \in \Sigma_k$ such that
	\[
		\HA + \Sigma_k\text{-}\mathbf{DNE} + \U_k\text{-}\mathbf{DNS} \vdash \varphi \leftrightarrow \varphi' ; 
	\]
if $\varphi \in \U_k^+$, then there exists a $\varphi' \in \Pi_k$ such that
	\[
		\HA + (\Pi_k \lor \Pi_k)\text{-}\mathbf{DNE} \vdash \varphi \leftrightarrow \varphi'. 
	\]
In addition, the authors studied in \cite{FK23} the conservation theorems on semi-classical arithmetic with respect to those classes.
The class $\E_k^+$ (resp.~$\U_k^+$) is intended to form the class of formulas which are classically equivalent to some $\Sigma_k$-formula (resp.~$\Pi_k$-formula).
In addition, as mentioned in \cite{ABHK04}, the class $\PF_k$ is intended to represent the set of $\Delta_{k+1}$-formulas, namely, formulas which is equivalent to some $\Sigma_{k+1}$-formula and also to some $\Pi_{k+1}$-formulas.
Note that every formula with quantifier occurrences is classified into exactly one of $\E_{k+1}$, $\U_{k+1}$ and $\PF_{k+1}$
as mentioned in \cite{ABHK04}.

There is, however, some room for discussion on the hierarchical classes.
Firstly, the classical transformation should be distinguished from the equivalence over a classical theory.
In fact, the class of formulas which are transformed into some formula in $\Sigma_{k}$ by the above mentioned procedure is different from the class of formulas which are equivalent to some formula in $\Sigma_{k}$ over a classical theory (cf.~Remark \ref{rem: non-computability of Sk(PA)}).
Secondly, despite the intention behind the definition of the class $\E_k^+$ (resp.~$\U_k^+$), the definition does not exclude the possibility that it does not cover all the formulas which are classically transformed into some $\Sigma_k$-formula (resp.~$\Pi_k$-formula).
Thus a proper justification for the classes is still missing.

Motivated by these issues, in this paper, we give a proper justification for the hierarchical classes.
In particular, we formalize the above mentioned procedure for prenex normalization and investigate the relation between the classes of prenex formulas and the hierarchical classes in \cite{ABHK04, FK21} modulo the transformation procedure.
Although classes $\E_k, \U_k, \F_k, \PF_k, \E_k^+, \U_k^+ $ and $ \F_k^+$ are studied in the context of arithmetic in \cite{ABHK04, FK21}, they can be defined in a general context.
In this paper, we reformulate the classes in a general language of a first-order theory.
Then we first show that a formula is in $\E_k^+$ (resp.~$\U_k^+$) if and only if it can be transformed into a formula in $\Sigma_k^+$ (resp.~$\Pi_k^+$) by the transformation procedure, where $\Sigma_k^+$ and $\Pi_k^+$ are cumulative variants of $\Sigma_k$ and $\Pi_k$ in the general first-order language, respectively.
Then it follows that a formula is in $\F_k^+$ if and only if it can be transformed into a formula in $\Sigma_{k+1}^+$ and also into a formula in $\Pi_{k+1}^+$ by the transformation procedure.
By the results for the cumulative classes, it also follows that non-cumulative classes $\E_k$, $\U_k$ and $\PF_k$ (except $\PF_0$) are the cumulative counterparts of $\Sigma_k$, $\Pi_k$ and $\Delta_{k+1}$ respectively modulo the transformation procedure (cf.~Theorem \ref{thm: main results}).

All of our proofs in this paper are purely syntactic.

\section{Preliminaries}
We work with a standard formulation of first-order theories with all the logical constants $\forall, \exists, \to, \land, \lor$ and $ \perp$ in the language. 
Note that $\neg \varphi$ and $\varphi \leftrightarrow \psi $ are the abbreviations of $(\varphi \to \perp )$ and $(\varphi \to \psi) \land (\psi \to \varphi)$ respectively in our context.
Throughout this paper, let $k$ be a natural number (possibly $0$).
The classes $\Sigma_k$ and $\Pi_k$ are defined as follows (cf.~\cite[pp.~142--143]{CK13}):
\begin{itemize}
	\item Let $\Sigma_0 $ and $ \Pi_0$ be the class of all quantifier-free formulas; 
	\item $\Sigma_{k+1} : = \{\exists x_1, \dots,  x_n \, \varphi \mid \varphi \in \Pi_k\}$;
	\item $\Pi_{k+1} : = \{\forall x_1, \dots, x_n\, \varphi \mid \varphi \in \Sigma_k\}$;
\end{itemize}
where $n\geq 1$.
Let $\FV(\varphi)$ denote the set of all free variables in $\varphi$.
Their cumulative variants $\Sigma_k^+$ and $\Pi_k^+$ are defined as follows:
\begin{itemize}
	\item $ \begin{displaystyle}
 \Sigma_{k}^+ : =\Sigma_{k} \cup \bigcup_{i<k} \Sigma_i \cup  \bigcup_{i<k} \Pi_i ;
 \end{displaystyle}$
\item $ \begin{displaystyle}
  \Pi_{k}^+ : =\Pi_{k} \cup \bigcup_{i<k} \Sigma_i \cup  \bigcup_{i<k} \Pi_i
   \end{displaystyle}$.
\end{itemize}
A formula $\varphi$ is in \textbf{prenex normal form} if $\varphi$ is in $\Sigma_k \cup \Pi_k$ for some $k$.

In the following, we reformulate classes $\E_k, \U_k, \F_k, \PF_k, \E_k^+, \U_k^+ $ and $ \F_k^+$ introduced in \cite{ABHK04, FK21} in our general context (namely, in the language of an arbitrary given first-order theory).
In \cite{ABHK04}, classes $\E_k, \U_k, \F_k $ and $ \PF_k$ are described informally in the context of first-order arithmetic. 
In this paper, we employ the formal definitions given in \cite[Definition 2.11]{FK21}.
Note that our definition of $\PF_0$ is different from that in \cite{ABHK04} where $\PF_0$ is the set of quantifier-free formulas.

An \textbf{alternation path} is a finite sequence of $+$ and $-$ in which $+$ and $-$ appear alternatively.
For an alternation path $s$, let $i(s)$ denote the first symbol of $s$ if $s \not \equiv \ang{\, }$ (empty sequence); $ \times$ if $s \equiv \ang{\, }$.
Let $s^{\perp}$ denote the alternation path which is obtained by switching $+$ and $-$ in $s$, and let $l(s) $ denote the length of $s$.
For a formula $\vp$, the set of alternation paths $\alt{\vp}$ of $\vp$ is defined as follows:
\begin{itemize}
    \item 
    If $\vp$ is quantifier-free, then $\alt{\vp} := \{ \ang{\, } \}$;
    \item
    Otherwise, $\alt{\vp}$ is defined inductively by the following clauses:
    \begin{itemize}
        \item
        If $\vp \equiv \vp_1 \land \vp_2$ or $\vp \equiv \vp_1 \lor \vp_2$, then $\alt{\vp} := \alt{\vp_1} \cup \alt{\vp_2}$;
        \item
        If $\vp \equiv \vp_1 \to \vp_2$, then $\alt{\vp} := \{ s^{\perp} \mid s \in \alt{\vp_1}\} \cup \alt{\vp_2}$;
\item
If $\vp \equiv \forall x \vp_1 $, then $\alt{\vp} :=\{s \mid s\in \alt{\vp_1} \text{ and } i(s)\equiv -\} \cup \{-s \mid s\in \alt{\vp_1} \text{ and } i(s)\not \equiv - \} $;
\item
If $\vp \equiv \exists x \vp_1 $, then $\alt{\vp} :=\{s \mid s\in \alt{\vp_1} \text{ and } i(s)\equiv + \} \cup \{+s \mid s\in \alt{\vp_1} \text{ and } i(s)\not \equiv + \} $.
    \end{itemize}
\end{itemize}
In addition, for a formula $\vp$, the degree $\degree{\vp}$ of $\vp$ is defined as 
$$\degree{\vp} := \max \{l(s) \mid s \in \alt{\vp}  \} .$$

\begin{defn}
\label{def: Classes}
Classes $\F_k, \U_k, \E_k , \PF_k, \F_k^+, \U_k^+ , \E_k^+$ and $\PF_k^+$ are defined as follows:
\begin{itemize}
    \item
    $\F_k := \{ \vp \mid \degree{\vp}=k  \}    ;\,\, \F_k^+ := \{ \vp \mid \degree{\vp}\leq k  \}   ;$
    \item
    $\U_0:=\E_0:=\F_0 \, (=\Sigma_0 =\Pi_0)$;
    \item
$\U_{k+1} := \{ \vp \in \F_{k+1} \mid i(s) \equiv - \text{ for all }s\in \alt{\vp} \text{ such that }l(s) =k+1 \}$;
\item
$\E_{k+1} := \{ \vp \in \F_{k+1} \mid i(s) \equiv + \text{ for all }s\in \alt{\vp} \text{ such that }l(s) =k+1  \}$;
\item
$\PF_k := \F_k\setminus (\E_k \cup \U_k)$ \, (note that $\PF_0 =\emptyset$);
    \item
$\begin{displaystyle}
\U_{k}^+ := \U_{k} \cup \bigcup_{i<k} \F_i ;\, \, \E_{k}^+ := \E_{k} \cup \bigcup_{i<k} \F_i
\end{displaystyle}
$.
\end{itemize}
\end{defn}

\begin{remark}
Every formula with quantifier occurrences is classified into exactly one of $\E_{k+1}$, $\U_{k+1}$ and $\PF_{k+1}$ for some $k$.
\end{remark}

The distinction between $\E_k$ and $\E_{k}^+$ (as well as that for $\U_k$ and $\U_{k}^+$) is normally redundant since they are equivalent over a standard theory (cf.~\cite[Lemma 4.6]{FK21}).
For our investigation, however, the distinction is crucial because we focus on forms of formulas without mentioning any derivability relation.

\begin{lem}\label{lem: Classification}
\noindent
\begin{enumerate}
\item
\label{item: Ek+1 cap Uk+1+ = empty}
$\E_{k+1} \cap \U_{k+1}^+ =\U_{k+1} \cap \E_{k+1}^+ = \emptyset$.
\item
\label{item: Fk+=}
$\F_k^+ = \E_{k+1}^+ \cap \U_{k+1}^+ = \PF_{k} \cup \E_{k}^+ \cup \U_{k}^+$.
\end{enumerate}
\end{lem}
\begin{proof}
\eqref{item: Ek+1 cap Uk+1+ = empty}:
Suppose that $\vp \in \E_{k+1} \cap \U_{k+1}^+$.
Since $\vp \in \E_{k+1}$, we have $\vp \in \F_{k+1}$ and $\degree{\vp} =k+1 $, and hence, $\vp\in \U_{k+1}$.
Then there exists $s\in \alt{\vp}$ such that $l(s)=k+1$ and $i(s)=-$, which contradicts $\vp \in \E_{k+1}$.
Thus we have shown that $\E_{k+1} \cap \U_{k+1}^+ =\emptyset$.
The proof of $\U_{k+1} \cap \E_{k+1}^+ = \emptyset$ is similar.

\eqref{item: Fk+=}:
$\F_k^+ \subseteq \E_{k+1}^+ \cap \U_{k+1}^+$ is trivial by definition.
If the inclusion is proper, we have $(\E_{k+1}^+ \cap \U_{k+1}^+)\setminus \F_k^+ \neq \emptyset$, and hence, $\E_{k+1} \cap \U_{k+1} \neq \emptyset$, which contradicts \eqref{item: Ek+1 cap Uk+1+ = empty}.
Thus we have shown that $\F_k^+ = \E_{k+1}^+ \cap \U_{k+1}^+$.
The equality $\F_k^+ = \PF_{k} \cup \E_{k}^+ \cup \U_{k}^+$ is trivial.
\end{proof}

The following lemma is the reformulation of \cite[Lemma 4.5]{FK21} in our general context.
The proof is exactly the same as for \cite[Lemma 4.5]{FK21}.
\begin{lem}
\label{lem: basic facts on our classes}
The following hold for all $k$.
\begin{enumerate}
    \item
    \label{item: EU and}
    A formula $\vp_1 \land \vp_2$ is in $\U_k^+$ (resp. $\E_k^+$) if and only if both of $\vp_1 $ and $\vp_2$ are in  $\U_k^+$ (resp. $\E_k^+$).
    \item
        \label{item: EU or}
    A formula $\vp_1 \lor \vp_2$ is in $\U_k^+$ (resp. $\E_k^+$) if and only if both of $\vp_1 $ and $\vp_2$ are in  $\U_k^+$ (resp. $\E_k^+$).
    \item
            \label{item: EU to}
A formula $\vp_1 \to \vp_2$ is in $\U_k^+$ (resp. $\E_k^+$) if and only if $\vp_1 $ is in $\E_k^+$ (resp. $\U_k^+$) and $\vp_2$ is in  $\U_k^+$ (resp. $\E_k^+$).
 
 \item
         \label{item: U forall}
A formula $\forall x \vp_1$ is in $\U_k^+$ if and only if $\vp_1$ is in $\U_k^+$.
 \item
          \label{item: E exists}
A formula $\exists x \vp_1$ is in $\E_k^+$ if and only if  $\vp_1$ is in $\E_k^+$.
  \item
         \label{item: E forall}
A formula $\forall x \vp_1$ is in $\E_{k+1}^+$ if and only if it is  in $\U_k^+$.
   \item
            \label{item: U exists}
A formula $\exists x \vp_1$ is in $\U_{k+1}^+$ if and only if it is in $\E_k^+$.
\end{enumerate}
\end{lem}

\begin{remark}
    \label{rem: non-computability of Sk(PA)}
    In contrast to that the classes in Definition \ref{def: Classes} are computable, arithmetical classes defined by using provable equivalence are not computable in general.
    For example, for each natural number $k$,
    $$\Sigma_k(\PA):=\{\vp \mid \text{there exists $\psi\in \Sigma_k$ such that $\FV (\vp) =\FV (\psi)$ and $\PA \vdash \vp \leftrightarrow \psi $} \}$$
is not computable:
Suppose that $\Sigma_k(\PA)$ is computable.
Then there exists a formula $\tau(x)$ such that
\begin{itemize}
\item
$\vp\in \Sigma_k(\PA)$ implies $\PA \vdash \tau (\ulcorner\vp\urcorner )$;
\item
$\vp\notin \Sigma_k(\PA)$ implies $\PA \vdash \neg \tau (\ulcorner\vp\urcorner )$;
\end{itemize}
where $\ulcorner\vp\urcorner $ is the G\"odel number of $\vp$.
Fix a sentence $\psi \notin \Sigma_k(\PA)$.
By the fixed point theorem (cf.~\cite[p.~54]{Boo93}), there exists a sentence $\xi$ such that
\begin{equation}
    \label{eq: u}
    \PA \vdash \xi \leftrightarrow \tau ( \ulcorner \xi \land \psi \urcorner ).
\end{equation}

Suppose $\xi \land \psi \in  \Sigma_k(\PA)$.
Then $\PA \vdash \tau (\ulcorner\xi \land \psi \urcorner )$.
By \eqref{eq: u}, we have $\PA \vdash \xi$, and hence, $\PA \vdash \psi \leftrightarrow \xi \land \psi$.
Since $\psi \notin \Sigma_k(\PA)$, we have $\xi \land \psi \notin \Sigma_k(\PA)$, which is a contradiction.

Suppose $\xi \land \psi \notin  \Sigma_k(\PA)$.
Then $\PA \vdash \neg \tau (\ulcorner\xi \land \psi \urcorner )$.
By \eqref{eq: u}, we have $\PA \vdash \neg \xi$, and hence, $\PA \vdash {0=1} \leftrightarrow \xi \land \psi$.
Then we have $\xi \land \psi \in \Sigma_k(\PA)$, which is a contradiction.
\end{remark}

\section{Prenex normalization}
In this section, we formalize the notion of prenex transformation and study the basic property of the transformation.
Then we show a basic theorem (cf.~Theorem \ref{thm: Prenex normalization}) on prenex normalization with respect to the transformation, which will be investigated in more detail in the next section.

\begin{defn}\label{DT}
Let $\varphi$ and $\psi$ be any formulas. 
We say that $\psi$ is a \textit{prenex transformation} of $\varphi$, written as $\varphi \DT \psi$, if $\varphi$ and $\psi$ match one of the eight rows in the following table: for some formulas $\xi$ and $\delta$, variables $x \notin \FV(\delta)$ and $y$ where $y$ does not appear in $\xi$, and a quantifier $Q\in \{\forall, \exists \}$.
\begin{table}[ht]
\begin{center}
\begin{tabular}{|l|c|c|}
\hline
& $\varphi$ & $\psi$ \\
\hline
\hline
1& $\exists x \xi(x) \to \delta$ & $\forall x(\xi(x) \to \delta)$ \\
\hline
2& $\forall x \xi(x) \to \delta$ & $\exists x(\xi(x) \to \delta)$ \\
\hline
3-$Q$ & $\delta \to Q x\, \xi(x)$ & $Q x\, (\delta \to \xi(x))$ \\
\hline
4-$Q$& $Q x \, \xi(x) \land \delta$ & $Q x\, (\xi(x) \land \delta)$ \\
\hline
5-$Q$& $\delta \land Q x \, \xi(x)$ & $Q x\, (\delta \land \xi(x))$ \\
\hline
6-$Q$& $Q x \, \xi(x) \lor \delta$ & $Q x\, (\xi(x) \lor \delta)$ \\
\hline
7-$Q$& $\delta \lor Q x \, \xi(x)$ & $Q x\, (\delta \lor \xi(x))$ \\
\hline
8-$Q$& $Q x \, \xi(x)$ & $Q y \, \xi(y)$ \\
\hline
\end{tabular}
\end{center}
\end{table}
\end{defn}

\begin{remark}
Note that $\varphi$ and $\psi$ such that  $\varphi \DT \psi$ are equivalent over first-order classical logic.
Thus $\DT$ is a classically valid transformation.
On the other hand, intuitionistic logic does not admit translations 2 and 3-$\exists$ as well as the converses of 6-$\forall$ and 7-$\forall$.
\end{remark}

\begin{defn}
We write $\varphi \DTA \psi$ if $\psi$ is obtained from $\varphi$ by repeating prenex transformations finitely many times to a subformula recursively.
More formally, we write $\varphi \DTA \psi$ if there exist a natural number $k$ and finite sequences $\varphi_0, \ldots, \varphi_{k-1}, \varphi_k$, $\xi_0, \ldots, \xi_{k-1}$, and $\delta_0, \ldots, \delta_{k-1}$ of formulas such that $\varphi \equiv \varphi_0$, $\psi \equiv \varphi_k$, and for any $i < k$, $\xi_i \DT \delta_i$ and $\varphi_{i+1}$ is obtained by replacing an occurrence of a subformula $\xi_i$ in $\varphi_i$ with $\delta_i$. 
\end{defn}

The following propositions are trivial. 

\begin{prop}
\label{prop: DTA is reflexive and transitive}
The binary relation $\DTA$ is reflexive and transitive. 
\end{prop}

\begin{prop}
\label{prop: basic properties of DTA}
Suppose $\varphi \DTA \psi$. 
Then, 
\begin{enumerate}
	\item
 \label{item: DTA-FV}
 $\FV(\varphi) = \FV(\psi)$. 
	\item
 \label{item: DTA-NQ}
 $\varphi$ and $\psi$ have the same number of quantifiers. 
	\item
 \label{item: DTA-NLC}
 $\varphi$ and $\psi$ have the same number of logical connectives. 
\end{enumerate}
\end{prop}

\begin{prop}
\label{prop: DTA is classically valid}
If $\varphi \DTA \psi$, then $\varphi$ is equivalent to $\psi$ in classical first-order predicate logic. 
\end{prop}

\begin{lem}\label{DTA}
If $\psi$ is obtained by replacing an occurrence of a subformula $\xi$ in $\varphi$ with $\delta$ such that $\xi \DTA \delta$, then $\varphi \DTA \psi$. 
\end{lem}
\begin{proof}
Fix an occurrence of a subformula $\xi$ in $\varphi$.
Since $\xi \DTA \delta$, there exist a natural number $k$ and finite sequences $\xi_0, \ldots, \xi_{k-1}, \xi_k$, $\rho_0, \ldots, \rho_{k-1}$, and $\eta_0, \ldots, \eta_{k-1}$ such that $\xi \equiv \xi_0$, $\delta \equiv \xi_k$, and for any $i < k$, $\rho_i \DT \eta_i$ and $\xi_{i+1}$ is obtained by replacing an occurrence of $\rho_i$ in $\xi_i$ with $\eta_i$. 

We define a finite sequence $\varphi_0, \ldots, \varphi_k$ of formulas inductively as follows: 
\begin{itemize}
	\item $\varphi_0 : \equiv \varphi$. 
	\item $\varphi_{i+1}$ is obtained by replacing the occurrence of $\xi_i$ in $\varphi_i$ with $\xi_{i+1}$. 
\end{itemize}
Since $\xi_{i+1}$ is obtained by replacing the occurrence of $\rho_i$ in $\xi_i$ with $\eta_i$, we find that $\varphi_{i+1}$ is also obtained by replacing the occurrence of $\rho_i$ in $\varphi_i$ with $\eta_i$. 

Also, one can show by induction that for each $i < k$, $\varphi_{i+1}$ is obtained from $\varphi$ by replacing the occurrence of $\xi$ in $\varphi$ with $\xi_{i+1}$. 
Then, we have $\varphi_k \equiv \psi$ because $\xi_k \equiv \delta$. 
Therefore, we obtain $\varphi \DTA \psi$. 
\end{proof}

We show that every formula is transformed into a formula in prenex normal form by applying $\DTA$.

\begin{defn}
For each formula $\varphi$, let $\PNF(\varphi)$ be the set $\{\psi \mid \varphi \DTA \psi$ and $\psi$ is in prenex normal form$\}$. 
\end{defn}

\begin{thm}
\label{thm: Prenex normalization}
For any formula $\varphi$, $\PNF(\varphi)  \neq \emptyset$.
\end{thm}
\begin{proof}
We prove the theorem by induction on the number of quantifiers contained.
If $\varphi$ contains no quantifiers, then $\varphi \in \PNF(\varphi)$ by Proposition \ref{prop: DTA is reflexive and transitive}. 

Suppose that the theorem holds for formulas containing at most $n$ quantifiers.
In what follows, we prove by induction on the number of logical connectives that the theorem holds for formulas containing exactly $n+1$ quantifiers.
Suppose also that the theorem holds for formulas containing less than $k$ logical connectives. 
Let $\varphi$ contain exactly $k$ logical connectives and exactly $n+1$ quantifiers. 

If $\varphi$ is of the form $Q x\, \psi$ for some $Q \in \{\forall, \exists\}$ and $\psi$, then $\psi$ contains exactly $n$ quantifiers. 
By the induction hypothesis, there exists a $\psi' \in \PNF(\psi)$ such that $\psi \DTA \psi'$. 
By Lemma \ref{DTA}, we obtain $Q x \, \psi \DTA Q x\, \psi'$. 
Therefore, $Q x\, \psi' \in \PNF(\varphi)$. 

Otherwise, $\varphi$ is of the form $\varphi_0 \circ \varphi_1$ for some $\circ \in \{\land, \lor, \to\}$ and $\varphi_0, \varphi_1$. 
We prove only the case that $\varphi_0$ contains at least one quantifier (if $\varphi_0$ contains no quantifier, then $\varphi_1$ contains at least one quantifier). 
Since $\varphi_0$ contains less than $k$ logical connectives, by the induction hypothesis, there exists a $\varphi_0' \in \PNF(\varphi_0)$ such that $\varphi_0 \DTA \varphi_0'$. 
By Lemma \ref{DTA}, we obtain $\varphi \DTA \varphi_0' \circ \varphi_1$. 
Since $\varphi_0'$ contains at least one quantifier, by Proposition \ref{prop: basic properties of DTA}.\eqref{item: DTA-NQ}, $\varphi_0'$ is of the form $Q x\, \psi_0'(x)$. 
Let $y$ be any variable not occurring in $\psi_0'$ and $\varphi_1$. 
Then, by transformation $8$-$Q$, we have $Q x\, \psi_0'(x) \DT Q y\, \psi_0'(y)$. 
Also, for some appropriate quantifier $Q'$, we have $Q y\, \psi_0'(y) \circ \varphi_1 \DT Q' y\, (\psi_0'(y) \circ \varphi_1)$ by transformations $1,2$, $4$-$Q$ or $6$-$Q$, and hence, $\varphi \DTA Q' y\, (\psi_0'(y) \circ \varphi_1)$
by Lemma \ref{DTA}.
Note that $\psi_0'(y) \circ \varphi_1$ contains exactly $n$ quantifiers by Proposition \ref{prop: basic properties of DTA}.\eqref{item: DTA-NQ}.
Then, by the induction hypothesis, there exists a formula $\xi$ in prenex normal form such that $(\psi_0'(y) \circ \varphi_1) \DTA \xi$. 
By Lemma \ref{DTA} and Proposition \ref{prop: DTA is reflexive and transitive}, we have $\varphi \DTA Q' y\, \xi$. 
Therefore, $Q' y\, \xi \in \PNF(\varphi)$. 
\end{proof}

\begin{cor}[Prenex normal form theorem]
For a classical first-order theory ${\bf T}$ and a ${\bf T}$-formula $\vp$,  there exists a ${\bf T}$-formula $\vp'$ in prenex normal form such that $\FV(\vp) =\FV(\vp')$ and ${\bf T} \vdash \vp \leftrightarrow \vp'$.
\end{cor}
\begin{proof}
By Theorem \ref{thm: Prenex normalization} and Proposition \ref{prop: DTA is classically valid}.
\end{proof}

\section{Characterization of the hierarchical classes}
In this section, we first prove a hierarchical version of Theorem \ref{thm: Prenex normalization} for cumulative classes $\E_k^+$ and $\U_k^+$ of formulas.
Then we also investigate the converse direction and prove that if a formula $\vp$ can be transformed into some formula in $\Sigma_k^+$ (resp. $\Pi_k^+$), then $\vp $ is in $\E_k^+$ (resp. $\U_k^+$).
Then it follows that $\E_k^+$ and $\U_k^+$ are exactly the classes whose formulas are transformed into $\Sigma_k^+$ and $\Pi_k^+$, respectively.
Using those characterizations for $\E_k^+$ and $\U_k^+$, we also have reasonable characterizations for $\F_k^+$, $\E_k$, $\U_k$ and $\PF_k$ (cf.~Theorem \ref{thm: main results}).
In what follows, we sometimes use Propositions  \ref{prop: DTA is reflexive and transitive} and \ref{prop: basic properties of DTA} and Lemma \ref{DTA} without mention.

Firstly, we prove a basic lemma concerning logical connectives, quantifiers and the relation $\DTA$.

\begin{lem}\label{Connectives1}
Let $\circ \in \{\land, \lor\}$. 
\begin{description}
	\item [{\rm(A)}] If $\varphi \in \Sigma_{k+1}^+$, $\psi \in \Sigma_{k+1}^+$, and $k + 1 = \max \{\deg(\varphi), \deg(\psi)\}$, then there exists a $\sigma \in \Sigma_{k+1}$ such that $\varphi \circ \psi \DTA \sigma$. 

	\item [{\rm(B)}] If $\varphi \in \Pi_{k+1}^+$, $\psi \in \Pi_{k+1}^+$, and $k + 1 = \max \{\deg(\varphi), \deg(\psi)\}$, then there exists a $\pi \in \Pi_{k+1}$ such that $\varphi \circ \psi \DTA \pi$. 

	\item [{\rm(C)}] If $\varphi \in \Sigma_{k+1}$ and $\psi \in \Pi_{k+1}$, then there exist $\sigma \in \Sigma_{k+2}$ and $\pi \in \Pi_{k+2}$ such that $\varphi \circ \psi \DTA \sigma$ and $\varphi \circ \psi \DTA \pi$. 

	\item [{\rm(D)}] If $\varphi \in \Pi_{k+1}$ and $\psi \in \Sigma_{k+1}$, then there exist $\sigma \in \Sigma_{k+2}$ and $\pi \in \Pi_{k+2}$ such that $\varphi \circ \psi \DTA \sigma$ and $\varphi \circ \psi \DTA \pi$. 

	\item [{\rm(E)}] If $\varphi \in \Sigma_{k+1}^+$, $\psi \in \Pi_{k+1}^+$, and $k + 1 = \max \{\deg(\varphi), \deg(\psi)\}$, then there exists a $\pi \in \Pi_{k+1}$ such that $(\varphi \to \psi) \DTA \pi$. 

	\item [{\rm(F)}] If $\varphi \in \Pi_{k+1}^+$, $\psi \in \Sigma_{k+1}^+$, and $k + 1 = \max \{\deg(\varphi), \deg(\psi)\}$, then there exists a $\sigma \in \Sigma_{k+1}$ such that $(\varphi \to \psi) \DTA \sigma$. 

	\item [{\rm(G)}] If $\varphi, \psi \in \Sigma_{k+1}$, then there exist $\sigma \in \Sigma_{k+2}$ and $\pi \in \Pi_{k+2}$ such that $(\varphi \to \psi) \DTA \sigma$ and $(\varphi \to \psi) \DTA \pi$. 

	\item [{\rm(H)}] If $\varphi, \psi \in \Pi_{k+1}$, then there exist $\sigma \in \Sigma_{k+2}$ and $\pi \in \Pi_{k+2}$ such that $(\varphi \to \psi) \DTA \sigma$ and $(\varphi \to \psi) \DTA \pi$. 

\end{description}
\end{lem}
\begin{proof}
We prove (A) and (B) simultaneously by induction on $k$. 
Suppose that (A) and (B) hold for $k' < k$. 
We give only a proof of (A), and (B) is proved similarly. 

Assume that $\varphi \in \Sigma_{k+1}^+$, $\psi \in \Sigma_{k+1}^+$, and $k+1 = \max \{\deg(\varphi), \deg(\psi)\}$. 
Then, $\varphi$ and $\psi$ are of the forms $\exists \vec{x} \varphi_0 (\vec{x})$ and $\exists \vec{y} \psi_0 (\vec{y})$, respectively. 
Here $\varphi_0$ and $\psi_0$ are $\Pi_k^+$ formulas and $k = \max \{\deg(\varphi_0), \deg(\psi_0)\}$. 
Also at least one of $\exists \vec{x}$ and $\exists \vec{y}$ is non-empty. 
Let $\vec{z}$ and $\vec{w}$ be any finite sequences of variables not occurring in $\varphi$ and $\psi$. 
Then, we have $\varphi \DTA \exists \vec{z} \varphi_0 (\vec{z})$ and $\psi \DTA \exists \vec{w} \psi_0  (\vec{w})$ by $8$-$\exists$. 
Also, we have $\exists \vec{z} \varphi_0 \circ \exists \vec{w} \psi_0 \DTA \exists \vec{z}  \exists \vec{w} (\varphi_0 \circ \psi_0)$ by $4$-$\exists$, $5$-$\exists$, $6$-$\exists$ and $7$-$\exists$.

If $k = 0$, by Lemma \ref{DTA}, we conclude $\varphi \circ \psi \DTA \exists \vec{z} \exists \vec{w} (\varphi_0 \circ \psi_0)$ and $\exists \vec{z} \exists \vec{w} (\varphi_0 \circ \psi_0) \in \Sigma_1$. 

If $k \geq 1$, then by the induction hypothesis, there exists a $\pi \in \Pi_k$ such that $\varphi_0 \circ \psi_0 \DTA \pi$. 
Then, by Lemma \ref{DTA}, $\varphi \circ \psi \DTA \exists \vec{z} \exists \vec{w} \pi \in \Sigma_{k+1}$. 

\medskip
(C): Suppose $\varphi \in \Sigma_{k+1}$ and $\psi \in \Pi_{k+1}$. 
Then, $\varphi$ is of the form $\exists \vec{x} \varphi_0 (\vec{x} )$ for some non-empty sequence $\vec{x}$ of variables and $\varphi_0 \in \Pi_k$. 
We have $\varphi \DTA \exists \vec{y}(\varphi_0 (\vec{y} ) \circ \psi)$ by $4$-$\exists$, $6$-$\exists$ and $8$-$\exists$. 
By (B), there exists a $\pi \in \Pi_{k+1}$ such that $\varphi_0(\vec{y} ) \circ \psi \DTA \pi$. 
By Lemma \ref{DTA}, $\varphi \circ \psi \DTA \exists \vec{y} \pi$ and $\exists \vec{y} \pi \in \Sigma_{k+2}$. 
The existence of a $\pi \in \Pi_{k+2}$ with $\varphi \circ \psi \DTA \pi$ is proved in a similar way with using (A). 

\medskip
(D) is proved as in the case (C) by using (A) and (B). 
Clauses (E), (F), (G) and (H) are proved similarly. 
\end{proof}

In what follows, for the sake of simplicity of description, for example, we refer to the first clause of Lemma \ref{Connectives1} simply as (A).

\begin{lem}\label{Connectives2}
Let $\circ \in \{\land, \lor\}$. 
\begin{enumerate}
	\item If $\varphi \in \Sigma_{k+1}^+$ and $\psi \in \Sigma_{k+1}^+$, then there exists a $\sigma \in \Sigma_{k+1}^+$ such that $\varphi \circ \psi \DTA \sigma$. 

	\item If $\varphi \in \Pi_{k+1}^+$ and $\psi \in \Pi_{k+1}^+$, then there exists a $\pi \in \Pi_{k+1}^+$ such that $\varphi \circ \psi \DTA \pi$. 

	\item If $\varphi \in \Pi_{k+1}^+$ and $\psi \in \Sigma_{k+1}^+$, then there exists a $\sigma \in \Sigma_{k+1}^+$ such that $(\varphi \to \psi) \DTA \sigma$. 

	\item If $\varphi \in \Sigma_{k+1}^+$ and $\psi \in \Pi_{k+1}^+$, then there exists a $\pi \in \Pi_{k+1}^+$ such that $(\varphi \to \psi) \DTA \pi$. 

\end{enumerate}
\end{lem}
\begin{proof}
1. Let $\varphi \in \Sigma_{k+1}^+$ and $\psi \in \Sigma_{k+1}^+$. 
If $k+1= \max \{\deg(\varphi), \deg(\psi)\}$, then by (A), there exists a $\sigma \in \Sigma_{k+1} \subseteq \Sigma_{k+1}^+$ such that $\varphi \circ \psi \DTA \sigma$. 

Suppose $k+1 > k_0 = \max \{\deg(\varphi), \deg(\psi)\}$. 
We may assume that $k_0 \neq 0$.
We prove only the case of $k_0 = \deg(\varphi)$. 
The case of $k_0 = \deg(\psi)$ is proved similarly.  
We distinguish the following four cases: 
\begin{itemize}
	\item If $\varphi \in \Sigma_{k_0}$ and $\psi \in \Sigma_{k_0}^+$, then by (A), there exists a $\sigma \in \Sigma_{k_0} \subseteq \Sigma_{k+1}^+$ such that $\varphi \circ \psi \DTA \sigma$. 

	\item If $\varphi \in \Pi_{k_0}$ and $\psi \in \Pi_{k_0}^+$, then by (B), there exists a $\pi \in \Pi_{k_0} \subseteq \Sigma_{k+1}^+$ such that $\varphi \circ \psi \DTA \pi$. 

	\item If $\varphi \in \Sigma_{k_0}$ and $\psi \in \Pi_{k_0}$, then by (C), there exists a $\sigma \in \Sigma_{k_0+1} \subseteq \Sigma_{k+1}^+$ such that $\varphi \circ \psi \DTA \sigma$. 

	\item If $\varphi \in \Pi_{k_0}$ and $\psi \in \Sigma_{k_0}$, then by (D), there exists a $\sigma \in \Sigma_{k_0+1} \subseteq \Sigma_{k+1}^+$ such that $\varphi \circ \psi \DTA \sigma$. 

\end{itemize}
Other clauses are proved in a similar way. 
\end{proof}

\begin{thm}
\label{MT1}\leavevmode
\begin{enumerate}
	\item If $\varphi \in \E_k^+$, then $\PNF(\varphi) \cap \Sigma_k^+  \neq \emptyset$.
	\item If $\varphi \in \U_k^+$, then $\PNF(\varphi) \cap \Pi_k^+  \neq \emptyset$.
\end{enumerate}
\end{thm}
\begin{proof}
We may assume that $k\geq 1$.
We prove clauses 1 and 2 simultaneously by induction on the structure of $\varphi$. 
If $\varphi$ is atomic, then $\varphi \in \PNF(\varphi) \cap \Sigma_0 \subseteq \PNF(\varphi) \cap \Sigma_k^+ \cap \Pi_k^+$. 

We suppose that the theorem holds for $\varphi_0$ and $\varphi_1$.

\begin{itemize}
	\item Case of $\varphi \equiv \varphi_0 \circ \varphi_1$ for $\circ \in \{\land, \lor\}$. 

	1. If $\varphi_0 \circ \varphi_1 \in \E_k^+$, then $\varphi_0, \varphi_1 \in \E_k^+$ by Lemma \ref{lem: basic facts on our classes}.\eqref{item: EU and}, \eqref{item: EU or}. 
By the induction hypothesis, there exist $\sigma_0, \sigma_1 \in \Sigma_k^+$ such that $\varphi_0 \DTA \sigma_0$ and $\varphi_1 \DTA \sigma_1$. 
By Lemma \ref{Connectives2}.(1), there exists a $\sigma \in \Sigma_k^+$ such that $\sigma_0 \circ \sigma_1 \DTA \sigma$. 
Then, $\varphi_0 \circ \varphi_1 \DTA \sigma_0 \circ \sigma_1 \DTA \sigma$. 

	2. The case of $\varphi_0 \circ \varphi_1 \in \U_k^+$ is proved similarly by using Lemma \ref{Connectives2}.(2).  

	\item Case of $\varphi \equiv \varphi_0 \to \varphi_1$. 

	1. If $\varphi_0 \to \varphi_1 \in \E_k^+$, then $\varphi_0 \in \U_k^+$ and $\varphi_1 \in \E_k^+$ by Lemma \ref{lem: basic facts on our classes}.\eqref{item: EU to}. 
By the induction hypothesis, there exist $\pi_0 \in \Pi_k^+$ and $\sigma_1 \in \Sigma_k^+$ such that $\varphi_0 \DTA \pi_0$ and $\varphi_1 \DTA \sigma_1$. 
By Lemma \ref{Connectives2}.(3), there exists a $\sigma \in \Sigma_k^+$ such that $(\pi_0 \to \sigma_1) \DTA \sigma$. 
Then, $(\varphi_0 \to \varphi_1) \DTA (\pi_0 \to \sigma_1) \DTA \sigma$. 

	2. The case of $\varphi_0 \to \varphi_1 \in \U_k^+$ is proved similarly by using Lemma \ref{Connectives2}.(4).  

	\item Case of $\varphi \equiv \exists x \varphi_0$. 

	1. If $\exists x \varphi_0 \in \E_k^+$, then $\varphi_0 \in \E_k^+$ by Lemma \ref{lem: basic facts on our classes}.\eqref{item: E exists}. 
By the induction hypothesis, there exists a $\sigma_0 \in \Sigma_k^+$ such that $\varphi_0 \DTA \sigma_0$. 
Then, $\exists x \varphi_0 \DTA \exists x \sigma_0$ and $\exists x \sigma_0 \in \Sigma_k^+$. 

	2. If $\exists x \varphi_0 \in \U_k^+$, then $\exists x \varphi_0 \in \E_{k-1}^+$ by Lemma \ref{lem: basic facts on our classes}.\eqref{item: U exists}. 
	Then $\varphi_0 \in \E_{k-1}^+$.
By the induction hypothesis, there exists a $\sigma_{0} \in \Sigma_{k-1}^+ $ such that $\varphi_0 \DTA \sigma _{0}$, and hence, $\exists x \varphi_{0} \DTA \exists x \sigma_{0} \in\Sigma_{k-1}^{+} \subseteq \Pi_k^+$.

	\item Case of $\varphi \equiv \forall x \varphi_0$. \\
	This is proved similarly as in the case of $\exists$. \qedhere
\end{itemize}
\end{proof}

Secondly, we show the converse assertions of Theorem \ref{MT1}.

\begin{lem}\label{LemA}
Suppose $\varphi \DT \psi$. 
\begin{enumerate}
	\item If $\psi \in \E_k^+$, then $\varphi \in \E_k^+$. 
	\item If $\psi \in \U_k^+$, then $\varphi \in \U_k^+$. 
\end{enumerate}
\end{lem}
\begin{proof}
This lemma is proved by distinguishing the cases of the rows in the table in Definition \ref{DT} to which $\varphi$ and $\psi$ match.
In each case, we use the assertions in Lemma \ref{lem: basic facts on our classes} multiple times.
We prove only the case corresponding to the first row, and the other cases are proved in a similar way. 

For $x \notin \FV(\delta)$, suppose that $\varphi$ and $\psi$ are of the forms $\exists x \xi(x) \to \delta$ and $\forall x(\xi(x) \to \delta)$, respectively. 
	
	1. If $\forall x(\xi(x) \to \delta) \in \E_k^+$, then $k \geq 2$ and $\forall x(\xi(x) \to \delta) \in \U_{k-1}^+$.
	Since $\xi(x) \to \delta \in \U_{k-1}^+$,
 we have $\xi(x) \in \E_{k-1}^+$ and $\delta \in \U_{k-1}^+$.
	Then, $\exists x \xi(x) \in \E_{k-1}^+$, and hence $\exists x \xi(x) \to \delta \in \U_{k-1}^+ \subseteq \E_k^+$. 
	
	2. If $\forall x(\xi(x) \to \delta) \in \U_k^+$, then $\xi(x) \to \delta \in \U_k^+$. 
	We obtain $\xi(x) \in \E_k^+$ and $\delta \in \U_k^+$. 
	Then, $\exists x \xi(x) \in \E_k^+$, and thus $\exists x \xi(x) \to \delta \in \U_k^+$. 
\end{proof}

\begin{lem}\label{LemB}
Suppose that $\psi$ is obtained by replacing an occurrence of $\xi$ in $\varphi$ as a subformula with $\delta$ such that $\xi \DT \delta$. 
\begin{enumerate}
	\item If $\psi \in \E_k^+$, then $\varphi \in \E_k^+$. 
	\item If $\psi \in \U_k^+$, then $\varphi \in \U_k^+$. 
\end{enumerate}
\end{lem}
\begin{proof}
We prove the lemma by induction on the structure of $\varphi$. 
If $\varphi$ is atomic, then $\varphi$ is the unique subformula of $\varphi$.
Since there is no $\xi$  such that $\xi \DT \vp$, we are done.

Suppose that the theorem holds for $\varphi_0$ and $\varphi_1$.
By Lemma \ref{LemB}, we may assume that $\xi$ is a proper subformula of $\varphi$. 

\begin{itemize}
	\item Case of $\varphi \equiv \varphi_0 \circ \varphi_1$ for $\circ \in \{\land, \lor\}$. 
	
	Suppose $\psi \in \E_k^+$ (resp.~$\U_k^+$). 
	Since $\xi$ is a proper subformula of $\varphi$, $\xi$ is a subformula of either $\varphi_0$ or $\varphi_1$.
 If $\xi$ is a subformula of $\varphi_0$, then $\psi$ is of the form $\psi_0 \circ \varphi_1$, where $\psi_0$ is obtained by replacing an occurence of $\xi$ in $\varphi_0$ with $\delta$. 
	By Lemma \ref{lem: basic facts on our classes}.\eqref{item: EU and},\eqref{item: EU or}, $\psi_0$ and $\varphi_1$ are in $\E_k^+$ (resp.~$\U_k^+$). 
	By the induction hypothesis, we have $\varphi_0 \in \E_k^+$ (resp.~$\U_k^+$). 
	Hence, $\varphi_0 \circ \varphi_1 \in \E_k^+$ (resp.~$\U_k^+$). 
	The case that $\xi$ is a subformula of $\varphi_1$ is proved in a similar way. 

	\item Case of $\varphi \equiv \varphi_0 \to \varphi_1$. 

	We only give a proof for the case that $\psi \in \E_k^+$ and $\xi$ is a subformula of $\varphi_0$. 
	The other cases are proved similarly. 
	In this case, $\psi$ is of the form $\psi_0 \to \varphi_1$, where $\psi_0$ is obtained by replacing an occurence of $\xi$ in $\varphi_0$ with $\delta$. 
	By Lemma \ref{lem: basic facts on our classes}.\eqref{item: EU to}, $\psi_0 \in \U_k^+$ and $\varphi_1 \in \E_k^+$. 
	By the induction hypothesis, we have $\varphi_0 \in \U_k^+$. 
	Hence, $\varphi_0 \to \varphi_1 \in \E_k^+$. 
		
	\item Case of $\varphi \equiv \exists x \varphi_0$. 
	
	Let $\psi_0$ be the formula obtained from $\varphi_0$ by replacing an occurrence of $\xi$ in $\varphi_0$ with $\delta$. 
	Then, $\psi$ is of the form $\exists x \psi_0$. 
	Suppose $\psi \in \E_k^+$ (resp.~$\U_k^+$). 
	By Lemma \ref{lem: basic facts on our classes}.\eqref{item: E exists} (resp. Lemma \ref{lem: basic facts on our classes}.\eqref{item: U exists}), $\psi_0 \in \E_k^+$ (resp.~$\E_{k-1}^+$). 
	By the induction hypothesis, we obtain $\varphi_0 \in \E_k^+$ (resp.~$\E_{k-1}^+$). 
	Hence, $\exists x \varphi_0 \in \E_k^+$ (resp.~$\E_{k-1}^+ \subseteq \U_k^+$). 
	
	\item Case of $\varphi \equiv \forall x \varphi_0$. 
	
	This is proved similarly as in the case of $\exists$ with using Lemma \ref{lem: basic facts on our classes}.\eqref{item: U forall},\eqref{item: E forall}.
 . \qedhere
\end{itemize}

\end{proof}

\begin{lem}\label{LemC}
Suppose that $\varphi \DTA \psi$. 
\begin{enumerate}
	\item If $\psi \in \E_k^+$, then $\varphi \in \E_k^+$. 
	\item If $\psi \in \U_k^+$, then $\varphi \in \U_k^+$. 
\end{enumerate}
\end{lem}
\begin{proof}
Immediate from Lemma \ref{LemB} and the definition of $\varphi \DTA \psi$.
Note that if $\varphi \DTA \psi$ with a quantifier-free formula $\psi$, then $\varphi \equiv \psi$.
\end{proof}

\begin{thm}\label{MT2}\leavevmode
\begin{enumerate}
	\item If $\PNF(\varphi) \cap \Sigma_k^+ \neq \emptyset$, then $\varphi \in \E_k^+$. 
	\item If $\PNF(\varphi) \cap \Pi_k^+  \neq \emptyset$, then $\varphi \in \U_k^+$. 
\end{enumerate}
\end{thm}
\begin{proof}
1. Suppose $\psi \in \PNF(\varphi) \cap \Sigma_k^+$. 
Then, $\varphi \DTA \psi$ and $\psi \in \Sigma_k^+ \subseteq \E_k^+$. 
By Lemma \ref{LemC}.(1), we obtain $\varphi \in \E_k^+$. 

2 is proved in a similar way. 
\end{proof}

By Theorems \ref{MT1} and \ref{MT2}, we obtain characterizations of classes $\E_k^+$, $\U_k^+$, $\F_k^+$, $\E_k$, $\U_k$, $\PF_k$ in terms of the prenex normalization procedure as follows:

\begin{thm}[Main Theorem]\label{thm: main results}
\leavevmode
\begin{enumerate}
	\item
 \label{item: Ek+}
 $\varphi \in \E_k^+$ if and only if $\PNF(\varphi) \cap \Sigma_k^+  \neq \emptyset$.
	\item
  \label{item: Uk+}
  $\varphi \in \U_k^+$ if and only if $\PNF(\varphi) \cap \Pi_k^+  \neq \emptyset$.
 \item
 \label{item: Fk+}
 $\varphi \in \F_k^+$ if and only if $\PNF(\varphi) \cap \Sigma_{k+1}^+  \neq \emptyset$ and $\PNF(\varphi) \cap \Pi_{k+1}^+  \neq \emptyset$.
 \item
 \label{item: Ek}
 $\varphi \in \E_{k+1}$ if and only if $\PNF(\varphi) \cap \Sigma_{k+1}^+  \neq \emptyset$ and $\PNF(\varphi) \cap \Pi_{k+1}^+  = \emptyset$, equivalently, $\PNF(\varphi) \cap \Sigma_{k+1}  \neq \emptyset$ and $\PNF(\varphi) \cap \Pi_{k+1}^+  = \emptyset$.
 \item
 \label{item: Uk}
 $\varphi \in \U_{k+1}$ if and only if $\PNF(\varphi) \cap \Pi_{k+1}^+  \neq \emptyset$ and $\PNF(\varphi) \cap \Sigma_{k+1}^+  = \emptyset$, equivalently, $\PNF(\varphi) \cap \Pi_{k+1}  \neq \emptyset$ and $\PNF(\varphi) \cap \Sigma_{k+1}^+  = \emptyset$.
\item
 \label{item: PFk}
 $\varphi \in \PF_k$ if and only if $\PNF(\varphi) \cap \Sigma_{k+1}^+  \neq \emptyset$, $\PNF(\varphi) \cap \Pi_{k+1}^+  \neq \emptyset$ and $\PNF(\varphi) \cap (\Sigma_k^+ \cup \Pi_k^+)  = \emptyset$, equivalently, $\PNF(\varphi) \cap \Sigma_{k+1}  \neq \emptyset$, $\PNF(\varphi) \cap \Pi_{k+1}  \neq \emptyset$ and $\PNF(\varphi) \cap (\Sigma_k^+ \cup \Pi_k^+)  = \emptyset$.
\end{enumerate}
\end{thm}
\begin{proof}
Clauses \eqref{item: Ek+} and \eqref{item: Uk+} are immediate from Theorems \ref{MT1} and \ref{MT2}.

\medskip
Clause \eqref{item: Fk+} follows from clauses \eqref{item: Ek+} and \eqref{item: Uk+} since $\F_k^+ = \E_{k+1}^+ \cap \U_{k+1}^+$ (cf.~Lemma \ref{lem: Classification}.\eqref{item: Fk+=}).

\medskip
\eqref{item: Ek}:
By Lemma \ref{lem: Classification}.\eqref{item: Ek+1 cap Uk+1+ = empty}, we have that $\vp \in \E_{k+1}$ if and only if $\vp \in \E_{k+1}^+$ and $\vp\notin \F_k^+$ if and only if $\vp \in \E_{k+1}^+$ and $\vp\notin \U_{k+1}^+$.
Then, by clauses \eqref{item: Ek+} and \eqref{item: Uk+}, we have that  $\vp \in \E_{k+1}$ if and only if  $\PNF(\varphi) \cap \Sigma_{k+1}^+  \neq \emptyset$ and $\PNF(\varphi) \cap \Pi_{k+1}^+ =\emptyset$.
The latter is equivalent to that $\PNF(\varphi) \cap \Sigma_{k+1}  \neq \emptyset$ and $\PNF(\varphi) \cap \Pi_{k+1}^+  = \emptyset$ since $\Sigma_{k+1}^+ \setminus \Pi_{k+1}^+ =\Sigma_{k+1}$.

\medskip
Clause \eqref{item: Uk} is proved similarly to  \eqref{item: Ek}.

\medskip
Clause \eqref{item: PFk} is immediate from clauses \eqref{item: Fk+}, \eqref{item: Ek+} and \eqref{item: Uk+} since $\PF_k = \F_k^+ \setminus (\E_k^+ \cup \U_k^+)$ (cf.~Lemma \ref{lem: Classification}.\eqref{item: Fk+=}).
The last equivalence is trivial since $\Sigma_{k+1}^+ \setminus (\Sigma_k^+ \cup \Pi_k^+) = \Sigma_{k+1}$ and $\Pi_{k+1}^+ \setminus (\Sigma_k^+ \cup \Pi_k^+) = \Pi_{k+1}$.
\end{proof}

\begin{remark}
In Theorem \ref{thm: main results}, the characterizations for $\E_0$ and $\U_0$ are contained not in clauses \eqref{item: Ek} and \eqref{item: Uk} but in clauses \eqref{item: Ek+} and \eqref{item: Uk+} respectively since $\E_0=\E_0^+$ and $\U_0=\U_0^+$.
In addition, clause \eqref{item: PFk} does not hold for $\PF_0$ if one defines $\PF_0$ as the class of quantifier-free formulas as in \cite{ABHK04}.
\end{remark}

\section{Summary}
Theorem \ref{thm: main results}, which is our main theorem, reveals the following:
\begin{itemize}
\item
A formula is in $\E_k^+$ (resp.~$\U_k^+$) if and only if it can be transformed into a formula in $\Sigma_k^+$ (resp.~$\Pi_k^+$) with respect to $\DTA$.
\item
A formula is in $\F_k^+$ if and only if it can be transformed into a formula in $\Sigma_{k+1}^+$ and also into a formula in $\Pi_{k+1}^+$ with respect to $\DTA$.
\item
A formula is in $\E_{k+1}$ (resp.~$\U_{k+1}$) if and only if it can be transformed into a formula in $\Sigma_{k+1}$ (resp.~$\Pi_{k+1}$) but cannot be so for $\Pi_{k+1}^+$ (resp.~$\Sigma_{k+1}^+$) with respect to $\DTA$.
\item
A formula is in $\PF_{k}$ if and only if it can be transformed into a formula in $\Sigma_{k+1}$ and also into  a formula in $\Pi_{k+1}$ (resp.~$\Pi_{k+1}$) but cannot be so for $\Sigma_{k}^+ \cup \Pi_{k}^+$ (resp.~$\Sigma_{k+1}^+$) with respect to $\DTA$.
\end{itemize}
By this observation, the classification of formulas into $\E_k$, $\U_k$ and $\PF_k$ can be visualized as Figure \ref{fig:my_label}.

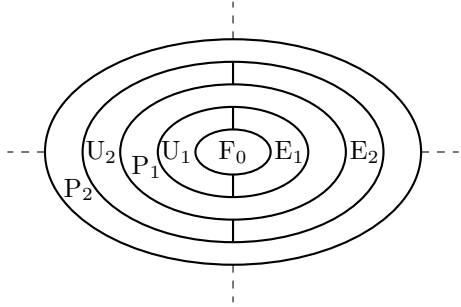
\begin{figure}[ht]
\centering
\begin{tikzpicture}
\node at (0,0) {$\F_0$};
\node at (-0.75,0) {$\U_1$};
\node at (0.75,0) {$\E_1$};
\node at (-1.15, -0.2) {$\PF_1$};
\node at (-1.75,0) {$\U_2$};
\node at (1.75,0) {$\E_2$};
\node at (-2.05, -0.5) {$\PF_2$};

\draw [thick] (0,0)circle [x radius=0.5,y radius=0.3];
\draw [thick] (0,0)circle [x radius=1,y radius=0.6];
\draw [thick] (0,0)circle [x radius=1.5,y radius=0.9];
\draw [thick] (0,0)circle [x radius=2,y radius=1.2];
\draw [thick] (0,0)circle [x radius=2.5,y radius=1.5];

\draw [thick] (0,0.3)--(0,0.6);
\draw [thick] (0,-0.3)--(0,-0.6);
\draw [thick] (0,0.9)--(0,1.2);
\draw [thick] (0,-0.9)--(0,-1.2);
\draw [dashed] (0,1.5)--(0,2);
\draw [dashed] (0,-1.5)--(0,-2);
\draw [dashed] (2.5,0)--(3.0,0);
\draw [dashed] (-2.5,0)--(-3.0,0);
\end{tikzpicture}
\caption{Hierarchical classification of formulas with respect to the prenex normalization}
    \label{fig:my_label}
\end{figure}

The difference between our Figure \ref{fig:my_label} and \cite[Figure 1]{ABHK04} is only in the position of $\PF_{k+1}$.
Our Figure \ref{fig:my_label} represents that
formulas in $\PF_k$ are outside of $\E_{k}^+ \cup \U_{k}^+$ which reflects the fact that a formula in $\PF_k$ cannot be transformed into a formula in  $\E_{k}^+ \cup \U_{k}^+$.
Since $\mathrm{C}_{k+1}:= \PF_k \cup \E_{k+1} \cup \U_{k+1}$ is the class of formulas which can be transformed into a formula in $\Sigma_{k+1} \cup \Pi_{k+1}$ but not so for  $\Sigma_{k} \cup \Pi_{k}$, one may think of $\mathrm{C}_{k+1}$ as the class of ``prenex degree'' $k+1$, which is based on the degree of prenex formulas into which the formula in question can be transformed with respect to $\DTA$.

\section*{Acknowledgements}
The authors thank Ulrich Kohlenbach for pointing them out that the definition of $\PF_0$ in \cite{ABHK04} is different from that in its preprint version, to which the authors referred in the previous version of this paper.
They also thank Danko Ilik for providing some information about related works.
The first author was supported by JSPS KAKENHI Grant Numbers JP19J01239, JP20K14354 and JP23K03205, and the second author by JP19K14586 and JP23K03200.

\bibliographystyle{plain}
\bibliography{sn-bibliography}

\end{document}